\documentclass{article}
\usepackage{amsmath}
\usepackage{amsthm}
\usepackage{amssymb}
\usepackage{dsfont}
\usepackage{stmaryrd}
\usepackage{comment}
\usepackage{authblk}
\usepackage{graphicx}
\usepackage[a4paper, margin=1in]{geometry}
\usepackage[utf8]{inputenc}
\usepackage{xcolor}

\providecommand{\keywords}[1]
{
  \small	
  \textbf{\textit{Keywords---}} #1
}

\newtheorem{assum}{Assumption}
\newtheorem{thm}{Theorem}
\newtheorem{lem}{Lemma}

\theoremstyle{remark}
\newtheorem{rmk}{Remark}

\providecommand{\diff}{\mathrm{d}}

\providecommand{\oh}{o}
\providecommand{\Oh}{\mathcal{O}}
\providecommand{\eps}{\varepsilon}
\providecommand{\MC}{\mathcal{M}}

\providecommand{\vol}{\operatorname{vol}}
\providecommand{\reals}{\mathbb{R}}
\providecommand{\eps}{\varepsilon}
\providecommand{\expec}{\mathbb{E}}
\providecommand{\1}{\mathds{1}}
\providecommand{\tr}{\operatorname{tr}}
\providecommand{\var}{\operatorname{Var}}
\providecommand{\vect}{\operatorname{vec}}
\providecommand{\diag}{\operatorname{diag}}
 \providecommand{\vM}{\vol(\MC)}

\newcommand{\red}[1]{\textcolor{red}{#1}}
\newcommand{\blue}[1]{\textcolor{blue}{#1}}

\usepackage[backend=bibtex,style=authoryear,natbib=true, url=false,isbn=false, doi=false, eprint=false]{biblatex} 
\renewbibmacro{in:}{}
\title{Quadratically Regularized Optimal Transport: \\  nearly optimal potentials and convergence of discrete Laplace operators}
\author[,$\mathsection$]{Gilles Mordant\footnote{G. Mordant gratefully acknowledges the funding by the DFG for CRC1456.}}
\author[,$\ddagger$]{Stephen Zhang\footnote{S. Zhang gratefully acknowledges funding from the Australian Government RTP Program.}}
\affil[$\mathsection$]{IMS, Universit\"at G\"ottingen}
\affil[$\ddagger$]{Mathematics \& Statistics, University of Melbourne}

\date{\today}

\begin{document}

\maketitle

\begin{abstract}
We consider the conjecture proposed in \citet{matsumoto2022beyond} suggesting that optimal transport with quadratic regularisation can be used to construct a graph whose discrete Laplace operator converges to the Laplace--Beltrami operator. We derive first order optimal potentials for the problem under consideration and find that the resulting solutions exhibit a surprising resemblance to the well-known Barenblatt--Prattle solution of the porous medium equation. Then, relying on these first order optimal potentials, we derive the pointwise $L^2$-limit of such discrete operators built from an i.i.d. random sample on a smooth compact manifold. Simulation results complementing the limiting distribution results are also presented.
\end{abstract}

{
}

\section{Introduction and main results}
\subsection{Discrete optimal transport with quadratic regularisation.}

In this entire paper, we will consider a $d$-dimensional compact smooth Riemannian manifold $(\MC,\rm g)$ isometrically embedded in $\reals^p$ via the embedding $\iota : \MC \hookrightarrow \reals^p$. In the sequel, we denote by $\iota_*$, the differential of this embedding.

Consider $\{ \iota(x_i) \}_{i=1}^N $ a set of points of the manifold embedded in the ambient Euclidean space. These can be random or deterministic. 
Then, let \[
C_{ij} = C(x_i, x_j) = \frac{1}{2} \| \iota(x_i) - \iota(x_j) \|_2^2
\] be the matrix of pairwise distances. Let $\mu = N^{-1} \sum_{i} \delta_{x_i}$ be the uniformly weighted empirical measure. The discrete quadratically regularised optimal transport (QOT) problem reads
\begin{align}
    \min_{\pi \in \Pi(\mu, \mu)} \langle C, \pi \rangle + \frac{\varepsilon}{2} \| \pi \|_2^2 &\Leftrightarrow \min_{\pi \in \Pi(\mu, \mu)} \| \pi + \varepsilon^{-1} C \|_2^2,
    \label{eq:qot}
\end{align}
where $\Pi(\mu, \mu) = \{ \pi : \pi \1 = \pi^\top \1 = \1 /N \}$ denotes the set of bistochastic couplings. Following \citet[Equation D]{lorenz2021quadratically}, we have the following duality result in terms of the dual potential $u$. 
\begin{align}
\label{eq: Dual}
    \sup_{u}\ \langle u, \mu \rangle - \frac{1}{4 \varepsilon} \left\| [u \oplus u - C]_+ \right\|_2^2,
\end{align}
where $[x]_+$ is the positive part of $x$. Denoting by $u^\star$ the optimal solution of~\eqref{eq: Dual} the relationship at optimality between the primal and dual variables is given by
\begin{align}
\label{eq: optimalPlan}
    \pi^\star = \frac{[u^\star \oplus u^\star - C]_+}{\varepsilon}.
\end{align}
As it is an optimal transport plan belonging to $\Pi(\mu,\mu)$, we recall the constraints 
\begin{equation}
\label{eq: Const}
    \sum_{i=1}^N \pi_{i,j}^\star = \frac{1}{N}, \qquad \forall j \in \{1, \ldots, N\},
\end{equation}
which are crucial to understand the optimal dual potentials.

\subsection{Discrete operators based on QOT.}

A consequence of the formula (\ref{eq: optimalPlan}) is that the optimal transport plan is sparse -- entries $\pi_{ij}^\star$ are identically zero whenever $C_{ij}$ becomes too large.
Because of this, \citet{matsumoto2022beyond} proposed to use the (rescaled) optimal transport plan as the adjacency matrix of a graph between pairs of points. The resulting weighted, undirected graph can in turn be used for downstream applications including semi-supervised learning, manifold learning, or dimensionality reduction in single cell RNA sequencing applications. The authors found that remarkable performance was achieved in these examples. More precisely, the weight matrix $W$ that they consider is given by 
\begin{equation}
\label{eq: GraphWeights}
W_{i,j} := \frac{\pi_{i,j}^\star}{\sum_{j=1}^N \pi_{i,j}^\star }.
\end{equation}
\citet{matsumoto2022beyond} further raise the question whether a discrete Laplace operator constructed from this matrix converges to a Laplace--Beltrami operator in the limit of infinitely many samples. This motivates the present paper.

\subsection{Main contributions and outline}

Our main contribution is twofold. First, we establish the correct asymptotic order of the potentials as a function of the regularisation parameter $\eps$ and the sample size $N$ in the discrete setting.  
This constitutes the content of Sections~\ref{sec: AnsatzCont} and~\ref{sec: Discrete}. Interestingly we find that the rates match with the solution of the porous medium equation, which is believed to be linked to optimal transport with a quadratic regularisation. Then, in Section~4, we then prove that, under suitable conditions, the discrete operator indeed can converge to the Laplace--Beltrami operator for random samples, as stated in Theorem~1. Section~5 considers the particular case of equidistant points on a circle. Section~6 can then be seen as empirical, finite sample size examples supporting the limits established.

\subsection{Manifold setting and notation}

Let us now describe a bit more the manifold setting that we consider for random samples. Let $X$ be a $p$-dimensional random variable whose range is supported on $\MC$. Let us further assume for simplicity that $X$ has a uniform distribution on the manifold, i.e., the density $\diff P(x) = \vol^{-1}(\MC) \diff x, \forall x \in \iota(\MC)$. Whenever we write $\expec$ or $\var$, unless otherwise denoted we mean to be with respect to $P$. Let us denote by $s(x)$, the scalar curvature of the manifold at $x$ and by $\mathbb{I}_x$ the second fundamental form of the isometric embedding $\iota$ at $x$. 
In the sequel, $\nabla$ will denote the covariant derivative while $\Delta$ will be the Laplace--Beltrami operator.
Further, set 
\[
\omega(x)=\frac{1}{\lvert S^{d-1}\rvert } \int_{S^{d-1}} \lVert \mathbb{I}_x(\theta, \theta) \rVert^{2} \diff \theta,
\]
as well as 
\[
\mathfrak{N}(x)=\frac{1}{\lvert S^{d-1} \rvert} \int_{S^{d-1}}  \mathbb{I}_x(\theta, \theta)  \diff \theta,
\]
where $S^{d-1}$ denotes the $(d-1)$-dimensional unit sphere in $T_x \MC$.

For the sake of simplicity, let us make the following assumptions. 

\begin{assum}
\label{assum: Rot}
The manifold $\MC$ is correctly shifted and rotated so that $\iota_* T_{x_0}\MC$ is spanned by $e_1, \ldots, e_d$.
\end{assum}

\begin{assum}
\label{assum: Assum2}
The manifold is properly rotated and translated so that $e_{d+1}, \ldots, e_p$ diagonalise the second fundamental form $\mathbb{I}_{x_0}$. 
\end{assum}

%\begin{assum}
%\label{assum: Assum3}
%Assume that \[
%\frac{1}{\lvert S^{d-1} \rvert} \int_{S^{d-1}} e_i^\top \mathbb{I}_{x_0}(\theta, \theta)  \diff \theta=0,
%\]
%for  $d+1\le i \le p$.
%\end{assum}

These two conditions are not particularly important, they just help simplify both notation and result statements.
Under these assumptions, we use the notation $\llbracket v_1, v_2\rrbracket$ for the vector $v$ whose $d$ first components are the vector $v_1$ and its $p-d$ last components are the vector $v_2$. The $p\times r$ matrix $\tilde{J}_{p,r}$ is then defined as 
\[
\tilde{J}_{p,r}:= \begin{pmatrix}
0_{p-r\times r} \\ I_{r\times r}
\end{pmatrix}
\]
%The last condition is more stringent as it is a mean curvature condition. It is for instance satisfied for locally affine manifolds or certain saddle points. As we shall see, this condition is necessary to avoid ``drift terms'' in the limit of the discrete operators. These drift terms are first-order differential terms.
In what follows, for some $\eta$ sufficiently small, define the fattened manifold in the embedded space by $\mathcal{N}$, i.e. the set $\mathcal{N}:= \{ x\in \reals^p: \inf_{y\in \MC} \lVert x-\iota (y) \rVert \le \eta\}$. 
We recall that the Laplace operator for a function $g$ on a certain embedded smooth manifold $\iota(\MC)$ is defined for $p \in \iota(\MC)$ as
\[
\Delta g( p) : = (\Delta_{\reals^p} g_{\text{ext}}) (p) 
\]
 with $\Delta_{\reals^p}$ the usual Laplace operator in the Euclidean space, $g_{\text{ext}}(x):= g ( \pi_\MC(x))$, $x\in \mathcal{N}$, and $\pi_\MC$ projects $\mathcal{N}$ onto the manifold.

%----------------------------------------------------------------------------
%
%                  Ansatz 
%
%----------------------------------------------------------------------------

\section{Ansatz for the potentials in the continuous case}
\label{sec: AnsatzCont}
An important question is understanding the relationship between the optimal potential $u^\star$ and the chosen regularisation parameter $\eps$. 
 For $x_0 \in \MC$,  it holds for $r$ sufficiently small\footnote{
For the entire paper, ``$r$ sufficiently small'' must be understood as $r$ being smaller than the injectivity radius of the manifold.
} that  
\begin{align}
\label{eq: B5}
\nonumber
\expec\left[ f(X; r) \1\{ \lVert X - \iota(x_0)\rVert\le r \}  \right] &= \frac{\lvert S^{d-1}\rvert}{d \vM} f(\iota(x_0); r)   r^d  \\ \nonumber 
&\qquad +
\frac{\lvert S^{d-1}\rvert}{d(d+2)}\Big(
\frac{1}{2\vM} \Delta f(x_0; r) +\frac{s(x_0)f(x_0; r)}{6\vM}  \\ 
&  \quad\qquad\qquad \qquad  \qquad \qquad + \frac{d(d+2)\omega(x_0)f(x_0; r)}{24 \vM}
\
\Big)r^{d+2} \\ \nonumber
 &\qquad +\Oh( f(x_0; r) r^{d+3}). 
\end{align}
This a slight variation of Lemma~B.5 of \citep{wu2018think} for a uniform density and in which the functions are allowed to depend on the parameter $r$. Because of this modification the asymptotic expansion has been slightly refined. 

Consider the continuous setting of the problem \eqref{eq:qot}: following \citet{lorenz2021quadratically} we write $\pi \in L^2(\MC)$ to be the density of a candidate transport plan w.r.t. product measure on $\MC \times \MC$, i.e. $\int \pi(x, y) \diff x \diff y = 1$. Then, the optimal transport plan $\pi^\star$ in the quadratically regularised problem must satisfy 
\[
    \frac{1}{\vM} \int_\MC \pi^\star(\iota(x_0), y) \diff y = \expec \left[ \pi^\star(\iota(x_0), X) \right] = \frac{1}{\vM^2}.
\]
Taking the relation $\pi^\star = \varepsilon^{-1} [u^\star \oplus u^\star - C]_+$ where $u^\star \in L^2(\MC)$ is the corresponding optimal dual potential, making the ansatz that $u^\star \sim \eps^\alpha$ and invoking \eqref{eq: B5} we have, for $\varepsilon \ll 1$,
\begin{align*}
\expec\left[ u^\star(\iota(x_0)) + u^\star(X) - \frac{1}{2} \| \iota(x_0) - X\|_2^2 \right]_+ &\sim \expec\left[ \left( \eps^\alpha - \lVert X - \iota(x_0)\rVert^2 \right) \1\{ \lVert X - \iota(x_0)\rVert\le \eps^{\alpha/2} \}  \right] \\
&\sim \frac{\lvert S^{d-1}\rvert}{d} \eps^{\tfrac{\alpha(d+1)}{2}}  + \Oh\left( \eps^{\tfrac{\alpha(d+2)}{2}}\right),
\end{align*}
(in the above multiplicative constants were dropped). The above quantity must behave asymptotically like $\varepsilon$ at leading order, and so matching exponents gives us
\[
\alpha+ \frac{d\alpha}{2} = 1 \Leftrightarrow \alpha = \frac{2}{2+d}.
\]
In the continuous case, the optimal potentials must thus behave in the first order like $\eps^\frac{2}{2+d}$.

There is a belief in the community that there should be some link between quadratically regularised optimal transport and a class of nonlinear partial differential equations known as the porous medium equation on $\reals^d$ for index $m = 2$ (see e.g. \citet{lavenant2018dynamical}), i.e., the equation
\[
\frac{\partial u}{\partial t }= \Delta(u^m), 
\]
where $u=u(x,t)$ and with an initial condition on $u$ at time $t=0.$ Starting from a Dirac mass of integral $\mathfrak{m}$ at the origin, the solution of the porous medium equation for $m =2$ is given by the Barenblatt-Prattle formula (\cite{vazquez2007porous}):
\[
u(x,t) = \max\left\{0, t^{-\frac{d}{2+d }}\left(\mathfrak{m} - \frac{ 1 }{4 (d+2)}\frac{ \lVert x\rVert^2 }{t^{\frac{2}{2+d}}}  \right)  \right\}. 
\]
or, rewriting terms, 
\[
u(x,t) = \max\left\{0, t^{-1}
\left(\mathfrak{m}\ t^{\frac{2}{2+d}}  - \frac{ 1 }{4 (d+2)} \lVert x\rVert^2   \right)  \right\}. 
\]
As the porous medium equation conserves mass, the integral of $u(x,t)$ over $\reals^d$ is $\mathfrak{m}$. A key property of the porous medium equation which distinguishes it from the standard diffusion equation is that the solution remains compactly supported. This is a property that also applies to the transport plans derived from quadratically regularized optimal transport (\cite{lorenz2021quadratically}). 

Perhaps closer to the theory of optimal transport, the porous medium equation of index $m$ can also be understood as the 2-Wasserstein gradient flow of the Tsallis entropy of order $m$ (see for example, the discussion in \citet{peyre2015entropic}). The Tsallis entropy generalizes the Gibbs entropy: for $m = 1$, it coincides with the Gibbs entropy, while for $m = 2$ it is corresponds to the squared $L_2$ norm of the density. It is remarkable that the squared $L_2$ norm is the functional that generates the porous medium equation as Wasserstein gradient flow, which is also the regularizing functional used in quadratically regularized optimal transport exhibiting analogous sparsity and scaling behaviour. Furthermore in the entropy regularized setting where $m = 1$, optimal transport enjoys the celebrated connection to a theory of large deviations for Brownian motions and the Schr\"odinger problem (\cite{leonard2013survey}). One interesting theoretical question would be whether similar connections could hold in more general cases, e.g. $m > 2$. 

Although our work does not formally establish the existence of such a connection, it is interesting that the same types of exponents appear and that the solution of the porous medium equation is so close in form to the solution of the quadratically regularised optimal transport problem.

%----------------------------------------------------------------------------
%
%                  The discrete case 
%
%----------------------------------------------------------------------------

\section{Optimal potentials in the discrete case}
\label{sec: Discrete}

We now turn the to study of the optimal rates for the potentials in the discrete case.
Let us slightly change the setting and consider a sample of size $N+1$ where one point, $x_0\in \MC$, is fixed and the remaining ones are an i.i.d.\ random sample on the manifold. Set $X_0=\iota(x_0)$. Relabel the sample points so that $\lVert X_0 - X_1 \rVert^2 \le \lVert X_0 - X_2 \rVert^2\le \ldots \le \lVert X_0 - X_N \rVert^2$.

\begin{rmk}[No loss of generality in choosing $X_0$]
\label{rmk: FixPoint}
In the results below, the same analysis has to be carried out for each point $X_i$. One can thus view our (convenient) choice of working with one distinguished, deterministic point as a conditioning on an arbitrary $X_i$. Still, as the expectations of the quantities for $X_0$ fixed are constants with uniformly decaying terms, the reasoning would apply for each $X_i$ using the tower property of conditional expectation.   
\end{rmk}

Using Lemma B.5. with $f = 1$, we get that the quantile function of the local distribution of squared distances at $x_0$ is approximately 

\[
p \mapsto \left(  \frac{p}{\lvert S^{d-1} \rvert d^{-1} \vol(\MC)^{-1}}\right)^{2/d}, 
\]
so that the duality constraint in the discrete problem is approximately 
\begin{equation}
\label{eq: ApproxDual}
 \frac{N+1}{\eps} \sum_{j=0}^N \left(u^\star(X_0)+ u^\star(X_j) - \left(  \frac{U_{(j:N)}}{\lvert S^{d-1} \rvert d^{-1} \vol(\MC)^{-1}}\right)^{2/d} \right)_+=1,   
\end{equation}
where $U_{(j:N)}$ is the $j$-th sorted element of an i.i.d.\ sample of size $N$ of random variables uniformly distributed on $[0,1]$ and we take $U_{(0:N)} = 0$. Note that although we write down all $N$ order statistics $U_{(j:N)}$ and the expression for the quantile function is only a good approximation for $p \ll 1$, as long as $u(x_0) + u(X_j)$ is small, only the first few terms will be nonzero.

We get the equivalent problem
\begin{equation}
\label{eq: ApproxDisc}
\frac{N+1}{\eps (N+1)^{2/d}}\ \kappa_d\ \sum_{j=0}^N \left( \tilde{u}(x_0)+ \tilde{u}(X_j) - U_{(j:N)}^{2/d} \right)_+=1, 
\end{equation}
where we have set
 \[\tilde{u}(\cdot) = \frac{(N+1)^{2/d}}{\kappa_d} u^\star(\cdot),\]
and
 \[
\kappa_d := \left( \frac{\vol(\MC) d}{\lvert S^{d-1}\rvert} \right)^{2/d}.
 \]
 Choose $k^{2/d} < 2 \tilde{u} \leq (k+1)^{2/d} $ and plug $\tilde{u}$ (choosing a constant approximation to the potential) as a choice for the potential in \eqref{eq: ApproxDisc}. 
 It yields, 
\begin{equation}
\label{eq: ConstDisc}
 \frac{N+1}{\eps (N+1)^{2/d}}\ \kappa_d\ \left[ 2(k+1)\tilde{u} - \sum_{j=1}^k U_{(j:N)}^{2/d} \right]=1.
\end{equation}

Let us turn to the size of the sum in  \eqref{eq: ConstDisc}.
First, basic calculations show that 
\[
\expec \left[ \left (U_{(j:N)}\right)^{2/d} \right] = \frac{\Gamma(\tfrac2d + j)\Gamma(N+1)}{\Gamma(\tfrac2d + N+1 )\Gamma( j)},
\]
So that understanding the problem \eqref{eq: ApproxDual}, even in expectation and for constant potentials is not so easy for $d>2$. 

For $d=1$, we get 
\[
\sum_{j=1}^k \expec \left (U_{(j:N)}\right )^{2/d} = \frac{1}{(N+2)(N+1)} \left [ \frac13  k(k+1)(k+1) \right],
\]
while for $d=2$, it holds that 
\[
\sum_{j=1}^k \expec \left (U_{(j:N)}\right )^{2/d} = \frac{1}{(N+1)} \left [ \frac12  k(k+1) \right].
\]
In general,  
\[
\sum_{j=1}^k \expec \left (U_{(j:N)}\right )^{2/d} \approx  (N+1)^{-2/d} \sum_{j=1}^k j^{2/d}\left(1+ \frac{2-d}{d^2 j} +\Oh\left(\frac{1}{d^2}\right)\right),
\]
so that the leading order is 
\[
(N+1)^{-2/d} \frac{d}{d+2} k^{\frac{d+2}{d}}.
\]
Equation~\eqref{eq: ConstDisc} then becomes, 
\[
\frac{\eps N^{2/d -1}}{\kappa_d} \approx \left(1 - \frac{d}{d+2}\right) k^{\frac{d+2}{d}} = \left( \frac{2}{d+2}\right) k^{\frac{d+2}{d}}
\]
so that
\[
\tilde u \approx \frac12   \left( \frac{2+d}{2\kappa_d}\right)^{\frac{
2}{d+2}} \eps^{\frac{2}{d+2}} N^{\frac{(2-d)2}{d(d+2)}} 
\]
and then 
\[
  u \approx  \kappa_d^{\frac{d}{d+2}} \left( \frac{d+2}{2}\right)^{\frac{
2  }{d+2} } \eps^{\frac{2}{d+2}} N^{\frac{(2-d)2}{d(d+2)}-2/d} =    \left(  \frac{\vol(\MC) d}{\lvert S^{d-1}\rvert} \ \frac{(d+2)}{2}\right)^{\frac{
2}{d+2}} \eps^{\frac{2}{d+2}} N^{\frac{-4}{(d+2)}}
\]
Finally, using this first order to approximately solve the equation yields that the optimal potential must behave like 
\[
K_{\eps, N}:=\left( \frac{2\kappa_d}{d+2}\right)^{\frac{
2}{d+2}} \eps^{\frac{2}{d+2}}N^{-\frac{4}{(d+2)}} =:C_d\eps^{\frac{2}{d+2}}N^{-\frac{4}{(d+2)}}
\]

These results provide a reasonable ansatz, still these are only approximations. 
We thus now assess the quality of this first order approximation of the solution by evaluating how the dual constraints are fulfilled when plugging-in the first-order approximation of the solution.

%--------------------------------------------------------
%  Validity of the derived finite sample rates
%--------------------------------------------------------
\subsection{Validity of the derived finite sample rate}
\label{sec: Validity}
We will use the function
\[
 f(y) = C_d \eps^{2/(d+2)} N^{-4/(d+2)} - \lVert \iota(x_0)- y \rVert^2  =  K_{\varepsilon, N} - \lVert \iota(x_0)- y \rVert^2 
\]
 and apply to the result of Lemma B.5 from \citet{wu2018think} to evaluate the constraints arising from the dual formulation of the problem, recall~\eqref{eq: optimalPlan} and~\eqref{eq: Const}. Doing so,  one gets 
\begin{align*}
&\expec \left( \frac{N+1}{\eps} \sum_{j=0}^N \left( K_{\eps,N} - \lVert X_j - \iota(x_0) \rVert^2 \right)_+ \right)  = \frac{N+1}{\varepsilon} \left\{ K_{\varepsilon, N} + N \expec \left[ f(X) \1\left\{ \| \iota(x_0) - X \| \leq K_{\varepsilon, N}^{1/2} \right\} \right] \right\} \\
  & \qquad = \frac{N+1}{\varepsilon} K_{\varepsilon, N} + \frac{N(N+1)}{\varepsilon} \frac{|S^{d-1}|}{d \vol(\MC)} K_{\varepsilon, N}^{1 + d/2} \\
    & \qquad\qquad+ \frac{N(N+1)}{\varepsilon} \frac{|S^{d-1}|}{d(d+2)\vol(\MC)} K_{\varepsilon, N}^{1 + d/2} \left[ -d + \frac{s(x_0) K_{\varepsilon, N}}{6} + \frac{d(d+2) \omega(x_0) K_{\varepsilon, N}}{24} \right] \\
    & \qquad\qquad+ \frac{N(N+1)}{\varepsilon} \Oh(K_{\varepsilon, N}^{\frac{5 + d}{2}} ) \\
      & \qquad = \frac{N+1}{\varepsilon} K_{\varepsilon, N} + \frac{N(N+1)}{\varepsilon} \frac{|S^{d-1}|}{ \vol(\MC)} K_{\varepsilon, N}^{1 + d/2} \left(\frac{1}{d} - \frac{1}{d+2} \right) + \Oh\left( \frac{N(N+1)}{\varepsilon} K_{\varepsilon, N}^{2 + d/2}  \right) 
\end{align*}

In the display above, terms have orders $\Oh(N \varepsilon^{-1} K_{\varepsilon, N}) = \Oh(\varepsilon^{-d/(d+2)} N^{(d-2)/(d+2)})$,  $\Oh(N^2 \varepsilon^{-1} K_{\varepsilon, N}^{2 + d/2})$ and we remark that $\Oh(N^2 \varepsilon^{-1} K_{\varepsilon, N}^{1 + d/2}) = \Oh(1)$. 
This latter term is the leading order. Then, to fulfill the constraint, we need that
\[ 
C_d^{\frac{2+d}{2}} \frac{|S^{d-1}|}{\vol(\MC)} \left( \frac{1}{d} - \frac{1}{d+2} \right) = 1
\]
so that 
\[
C_d=  \left( \frac{\vol(\MC)}{|S^{d-1}|} \frac{d(d+2)}{2}  \right)^{\frac{2}{d+2}},
\]
which matches with the expression above.
One gets that the chosen rate gives the correct constraint in expectation at the first order. 

 One can also rewrite the conditions that $\Oh(\varepsilon^{-d/(d+2)} N^{(d-2)/(d+2)})= \oh(1)$ as 
 $\varepsilon^{-d} N^{(d-2)} \to 0 $ and  the condition $\Oh(N^2 \varepsilon^{-1} K_{\varepsilon, N}^{2 + d/2})= \Oh(K_{\eps,N})= \oh(1)$  as $\varepsilon^{2} N^{-4} \to 0$. Note that the latter condition was already somewhat required to apply  Lemma B.5 from \citet{wu2018think}. 
Together, these results indicate that the asymptotic scaling on $\varepsilon$ is
 $N^{1 - 2/d} \ll \varepsilon \ll N^2$  for the constraints to be asymptotically fulfilled in expectation. 
 Note that in this analysis that $\varepsilon$ need not go to zero asymptotically. Rather, the need is for $\varepsilon$ to be asymptotically sufficiently small relative to $N$. See the remark below. % This is in contrast to the continuous setting, in which the limit is achieved for $\varepsilon \downarrow 0$. 

\begin{rmk}
 \label{rmk: Uniformly}
    In the developments above, the result holds uniformly in $x_0$ under quite mild assumptions as, for a closed\footnote{Recall that a manifold is closed if it is compact and without boundary. } and smooth manifold, the different kinds of curvatures appearing in the expansions are bounded, recall Remark~\ref{rmk: FixPoint}.
\end{rmk}
Because of Remark~\ref{rmk: Uniformly}, one can derive that all the constraints will asymptotically be fulfilled in expectation when replacing the sum of optimal potentials by $K_{\eps,N}$. 

\begin{rmk}
There is a difference in scaling between the discrete and continuous settings in our analysis -- to get empirical input distributions that are consistent with the continuous setting in the limit of large $N$, for samples $X_1, \ldots, X_N$ we take $\hat{\mu} = N^{-1} \sum_i \delta_{X_i}$ as the corresponding empirical distribution. Suppose $\pi_{ij}$ is an admissible coupling for such a discrete problem. Then $\pi$ is concentrated on the support of $\hat{\mu} \otimes \hat{\mu}$ and admits a density, $(\diff \pi / \diff \hat{\mu} \otimes \diff \hat{\mu})(X_i, X_j) = \pi_{ij} / N^2$. Then, note that the corresponding empirical entropy term would behave like
\begin{align*}
    H(\pi | \hat{\mu} \otimes \hat{\mu}) = \int \diff \pi \log\left( \frac{\diff \pi}{\diff \hat{\mu} \otimes \diff \hat{\mu}} \right) = \sum_{ij} \pi_{ij} \log\left( \frac{\pi_{ij}}{N^2} \right)
\end{align*}
Up to a constant, this is equal to the discrete entropy of $\pi$, i.e. $\sum_{ij} \pi_{ij} \log \pi_{ij}$. Thus, we expect no scaling behaviour between $N$ and the entropic regularizer. 

On the other hand, for the quadratic regularizer, one would have
\begin{align*}
    \| \pi \|^2_{\hat{\mu} \otimes \hat{\mu}} = \int \frac{\diff \pi}{\diff \hat{\mu} \otimes \diff \hat{\mu}} \diff \pi = \sum_{ij} \pi_{ij} \frac{\pi_{ij}}{N^2} = N^{-2} \| \pi \|^2. 
\end{align*}
Thus, there is the presence of a factor $N^{-2}$. This can be understood in that $N$ appears in the density of $\pi$ w.r.t. empirical product measure, which is lost as an additive term in the case of a log, but remains in the quadratic case. Thus, noting that $\varepsilon \| \pi \|_{\hat{\mu} \otimes \hat{\mu}}^2 = (\varepsilon N^{-2}) \| \pi \|_2^2$, it is apparent that the requirement that $\varepsilon \to 0$ in the continuous setting corresponds to $\varepsilon N^{-2} \to 0$ in the discrete setting. This is in agreement with the scaling we derived earlier.

\end{rmk}

%--------------------------------------------------------
%  Replacing the optimum by a uniform approximation
%--------------------------------------------------------
\subsection{Replacing the optimum by a uniform approximation}

Similarly to \citet[Lemma~3.1]{lorenz2021quadratically}, 
the Newton Hessian of the optimisation problem~\eqref{eq: Dual} is 
\[
\diag (\sigma \boldsymbol{1}_{N+1})
\] 
where 
\[
\sigma_{ij} = \begin{cases}
    1 & \text{ if } u_i + u_j -C_{ij} \geq 0,  \\
    0 & \text{ otherwise}.
\end{cases}
\]
Because of the constraints, the potential must be chosen such that $u_j>0, \forall j\le N+1$. 
The function to optimise is thus strictly concave for the set of such potentials and thus admits a unique optimum.

The update step in the semismooth Newton algorithm used in \citet[Appendix A]{matsumoto2022beyond}\footnote{Note that the measures in their paper are not probability measures, which explains the slight difference.} and originally developed in \citet{lorenz2021quadratically} takes the form 
\[
G_{i,i}^{-1} \left(\sum_j (K_{\eps,N} -c_{i,j})_+ - \frac{\eps}{N+1} \right), \quad 1\le i\le N, 
\]
when the regularisation parameter of the algorithm is set to zero.
From \eqref{eq: B5} again, it holds that 
\[
\expec G_{ii} = \Oh(N K_{\eps,N}^{d/2}).
\]
From this and the computations of Section~\ref{sec: Validity}, the expectation of the update step is of order
\begin{align*}
\frac{1}{N K_{\eps,N}^{d/2} } \left( \Oh(K_{\eps,N}) + \Oh\left( N K_{\eps,N}^{2+ d/2}\right)\right) &= \Oh \left ( 
\eps^{\frac{2-d}{d+2}} N^{\frac{-2(2-d) - (d+2)}{d+2}}\right) +\Oh(K_{\eps, N}^2) \\
&= \Oh\left( 
\eps^{\frac{2-d}{d+2}} N^{\frac{d-6 }{d+2}}\right)+ \Oh(K_{\eps, N}^2)\\
&= \Oh\left( 
\eps^{\frac{-d}{d+2}} N^{\frac{d-2 }{d+2}} \ \eps^{\frac{2}{d+2}} N^{\frac{-4 }{d+2}}\right)+ \Oh(K_{\eps, N}^2),
\end{align*}
which goes to zero in view of the conditions on $\eps$ and $N$ mentioned above.

\begin{rmk}[QOT, nonparametric statistics and optimality]
The form of the optimal transport plan is very much alike an Epanechnikov kernel, which  is very often used in nonparametric statistics. The latter kernel is 
\[
 u \mapsto \frac{\Gamma(2+d/2)}{\pi^{\frac{d}{2}}}(1 - u^\top u) \1_{\{ u^\top u \le 1\}}.
\]
Even though this statement is debated \citep[Section~1.2.4]{tsybakov2008introduction}, the Epanechnikov kernel is often claimed to be the optimal nonnegative kernel in terms of asymptotic MISE for the estimation of a twice differentiable density. Thus, the compact support and the fact that the optimal dual potential is a function\textemdash which is likely more adaptive to the data than a uniform bandwidth, might explain the outstanding performances observed in the examples of \citet{matsumoto2022beyond}.
\end{rmk}

%--------------------------------------------------------
%  Discrete measure on the circle
%--------------------------------------------------------

\section{Graphs Laplacians based on Quadratically Regularized OT}

\subsection{Limiting operators}

Before proving the main result, we state two useful lemmas. 
\begin{lem}
\label{lem: LocCov}
For fixed $x_0 \in \MC$, $r$ sufficiently small and $X$ uniformly distributed on $\MC$ under Assumptions~\ref{assum: Rot} and ~\ref{assum: Assum2}, it holds that 
\begin{eqnarray*}
\lefteqn{\expec\left( f(X;r) (X- \iota(x_0))  (X- \iota(x_0))^\top  \1\{ \lVert X - \iota(x_0)\rVert\le r \}  \right)} \\
& \hspace{35mm}  = \frac{\lvert S^{d-1}\rvert}{d(d+2) \vol(\MC)} f(\iota(x_0);r)  r^{d+2}  
 \left(
 \begin{pmatrix} 
 I_{d\times d} & 0 \\ 0 & 0
 \end{pmatrix} 
 + \Oh(r^2)\right)
 \end{eqnarray*}
\end{lem}
\begin{proof}
The proof follows along the same lines as Proposition~3.1 in \citet{wu2018think}.
\end{proof}

\begin{lem}
\label{lem: FourthOrd}
For fixed $x_0 \in \MC$, $r$ sufficiently small and $X$ uniformly distributed on $\MC$ under Assumptions~\ref{assum: Rot} and ~\ref{assum: Assum2}, it holds that 
\begin{align*}
\expec \left[
f(X;r) e_k^\top(X-\iota(x_0))(X-\iota(x_0))^\top e_l e_m^\top(X-\iota(x_0))(X-\iota(x_0))^\top e_n \1\{ \lVert X - \iota(x_0)\rVert\le r \}
\right]\\
=
\frac{f(\iota(x_0);r)}{(d+4)\vM} r^{d+4} C_{k,l,m,n} + \Oh(r^{d+5}),
\end{align*}
where 
\[
C_{k,l,m,n} = \int_{S^{d-1}} \langle \iota_*\theta , e_k \rangle
\langle \iota_*\theta , e_l \rangle
\langle \iota_*\theta , e_m \rangle
\langle \iota_*\theta , e_n \rangle \diff \theta
\]

\end{lem}
\begin{proof}
First set 
\[
\tilde{B}_r(x_0):=\iota^{-1}(B_r^{\reals^p}(\iota(x_0))\cap \iota(\MC)).
\]
Then, the quantity of interest can be written
\[
\mathcal{I}:=\frac{1}{\vM}\int_{\tilde{B}_r(x_0)} 
\langle \iota(y)-\iota(x_0) , e_k \rangle
\langle \iota(y)-\iota(x_0) , e_l \rangle
\langle \iota(y)-\iota(x_0) , e_m \rangle
\langle \iota(y)-\iota(x_0) , e_n \rangle f(y)  \diff V(y).
\]
Recalling that for $(t,\theta )\in [0, \infty)\times S^{d-1}$
\begin{align*}
 \iota \circ \exp_{x_0} (\theta t) -\iota(x_0) &= \iota_* \theta t +\Oh(t^2)\\
\tilde r &= r + \Oh(r^3)\\
\diff V(\exp_{x_0} (\theta t))&= t^{d-1} + \Oh(t^{d+1})\\
f(\exp_{x_0} (\theta t))&= f(x_0) + \Oh(t),
\end{align*}
it holds that 
\begin{align*}
\mathcal{I}&= \frac{1}{\vM} \int_{S^{d-1}}\int_0^{\tilde r} f(x_0) t^{d+3} \langle \iota_*\theta , e_k \rangle
\langle \iota_*\theta , e_l \rangle
\langle \iota_*\theta , e_m \rangle
\langle \iota_*\theta , e_n \rangle + \Oh(t^{d+4}) \diff t \diff \theta \\
& =\frac{f(x_0)}{\vM(d+4)} r^{d+4}\int_{S^{d-1}}\langle \iota_*\theta , e_k \rangle
\langle \iota_*\theta , e_l \rangle
\langle \iota_*\theta , e_m \rangle
\langle \iota_*\theta , e_n \rangle \diff \theta + \Oh(r^{d+5});
\end{align*}
as claimed.
\end{proof}

We can now state our main theorem. Note that we consider functions defined on the ambient space, $\reals^p$ as opposed to only on $\MC$, since in manifold learning $\MC$ is unknown and the operator will thus be applied to function on $\reals^p$.  
\begin{thm} 

Consider $g \in C^2$, $g: \reals^p \to \reals$. Denote by $Q_0$ the Hessian of $g$ at $x_0$. To simplify notation, set $X_0=x_0$. Take a sample $\{\iota(X_j)\}_{j=1}^N$ from the uniform distribution on $\MC$ which is embedded in $\reals^p$. Then, under Assumptions~\ref{assum: Rot}~and~\ref{assum: Assum2}, defining 
 \[
 \Delta^{OT} g(X_0) := \sum_{j=0}^N W_{0,j}^\eps \big( g(X_0) -g(X_j)\big), 
 \]
 with $W^\eps$ the approximate solution of the quadratically regularised OT problem, it holds that, 
 \[
- 2 K_{\eps, N}^{-1}  \  \Delta^{OT} g(X_0) \xrightarrow{L^2}  d L_0 \left\llbracket  0 ,
 \frac{\tilde{J}_{p,p-d}^\top \mathfrak{N}(x_0)}{2}
\right\rrbracket  +\frac12\tr \left[Q_0 \begin{pmatrix} 
 I_{d\times d} & 0 \\ 0 & 0
 \end{pmatrix} \right],
 \]
 provided that $K_{\eps, N} \to 0$ and $N/\eps \to 0$ when $N\to \infty$.
\end{thm}

\begin{proof}
As $g$ is twice differentiable, we can write 
\[
g(X_j) = g(x_0) + L_0 (x_0-X_j) + \frac12 (X_j -x_0)^\top Q_0  (X_j -x_0) + \oh\left ( \lVert x_0 -X_j \rVert^2 \right),
\]
where $L_0$ is the gradient of $g$ at $x_0$ and $Q_0$ is the Hessian of $g$ evaluated at $x_0$.
Plugging this result in the definition of $( \Delta^{OT} g)(x_0)$, we derive 
\begin{align*}
&( \Delta^{OT} g)(x_0)  \\
& \qquad 
= \sum_{j=1}^N W_{0,j}^\eps \left( - L_0 (X_0-X_j) - \frac12 (X_j -X_0)^\top Q_0  (X_j -X_0) + \oh\left ( \lVert X_0 -X_j \rVert^2 \right)
\right).
\end{align*}
We will split this sum into three terms and control each one separately. We will first consider the expectation and then the variance.\\
\underline{\textit{Step~1: Expectation.}}

Let us start with the second term 
\[
- \frac12\sum_{j=0}^N W_{0,j}^\eps  (X_j -X_0)^\top Q_0  (X_j -X_0) =: B
\]
The quantity in the above display is a scalar so that it is equal to its trace. Further, using the linearity and the cyclical property of the trace, it holds that
\begin{align*}
B &= - \frac12 \tr \left [ Q_0 \sum_{j=0}^N W_{0,j}^\eps    (X_j -X_0) (X_j -X_0)^\top \right ] \\
&= - \frac12 (N+1) \tr \left [ Q_0 \sum_{j=0}^N \frac{( K_{\eps, N} -c_{0,j})_+}{\eps} (X_j -X_0) (X_j -X_0)^\top  \right   ]. 
\end{align*} 
Using Lemma~\ref{lem: LocCov}, it holds that 
\begin{align*}
&\expec \sum_{j=1}^N \frac{( K_{\eps, N} -c_{0,j})_+}{\eps} (X_j -X_0) (X_j -X_0)^\top 
\\
& \qquad\qquad\qquad = 
\frac{N}{\eps} \frac{\lvert S^{d-1}\rvert}{d(d+2) \vol(\MC)} K_{\eps, N}^{\frac{d+2}{2}} K_{\eps, N}   
 \left(
 \begin{pmatrix} 
 I_{d\times d} & 0 \\ 0 & 0
 \end{pmatrix} 
 + \Oh(K_{\eps,N})\right).
\end{align*}
 
Let us now address the first term, i.e., 
\[
-L_0 \sum_{j=0}^N W_{i,j}^\eps   (X_0-X_j).
\]
The second part of Lemma~B.5 in \citet{wu2018think} reads, in our case,
\begin{align*}
&\expec \left[ (X-\iota(x_0)) f(X;r) \1\{ \lVert X - \iota(x_0)\rVert\le r \}  \right]  \\
& \qquad \qquad =  \frac{\lvert S^{d-1} \rvert}{(d+2)\vol(\MC)} \left\llbracket \frac{ J_{p,d}^\top\iota_*\nabla f(x_0;r)}{d} ,
\frac{f(x_0;r) \tilde{J}_{p,p-d}^\top \mathfrak{N}(x_0)}{2}
\right\rrbracket r^{d+2} + \Oh(r^{d+4}).
\end{align*}
It follows that 
\begin{align*}
&\expec \sum_{j=1}^N W_{0,j}^\eps   (X_0-X_j) \\
&\qquad\qquad = \frac{N(N+1)}{\eps} \frac{\lvert S^{d-1} \rvert}{(d+2)\vol(\MC)} \left\llbracket  0 ,
\frac{K_{\eps, N} \tilde{J}_{p,p-d}^\top \mathfrak{N}(x_0)}{2}
\right\rrbracket K_{\eps,N}^{\frac{d+2}{2}}   + \Oh\left( \frac{N(N+1)}{\eps}K_{\eps,N}^{(d+4)/2}\right).
\end{align*}

In view of the developments above, the expectation of the Taylor residual is negligible. \\
\underline{\textit{Step~2: Variance.}}

Let us deal with \[
-L_0 \sum_{j=1}^N W_{0,j}^\eps   (x_0-X_j).
\]
We have that 
\[\var\sum_{j=1}^N W_{0,j}^\eps   (x_0-X_j) = \frac{N(N+1)^2}{\eps^2} \var ( ( K_{\eps, N} -c_{0,j})_+ (X -x_0))
\]
Further, 
\begin{align*}
\var [ ( K_{\eps, N} -c(X,x_0))_+ (X -x_i)] &= \expec [( K_{\eps, N} -c(X,x_0))_+^2 (X -x_0)(X -x_0)^\top] \\
&- \expec [( K_{\eps, N} -c(X,x_0))_+ (X -x_i)] \expec [( K_{\eps, N} -c(X,x_0))_+ (X -x_0)^\top]\\
&= \frac{\lvert S^{d-1}\rvert}{d(d+2) \vol(\MC)} \kappa_d^{\frac{d+2}{2}} K_{\eps, N}^2  \eps
 N^{-2}  
 \left(
 \begin{pmatrix} 
 I_{d\times d} & 0 \\ 0 & 0
 \end{pmatrix} 
 + \Oh(K_{\eps,N})\right) \\
 &-  \frac{\lvert S^{d-1} \rvert^2}{(d+2)^2\vol^2(\MC)} 
 K_{\eps, N}^2
 vv^\top \kappa_d^{{d+2}}  \eps^2
 N^{-4}   + \Oh\left(K_{\eps,N}^{(d+4)/2}\eps N^{-2} K_{\eps,N}\right), 
\end{align*}
where $v$ are vectors that depend on the curvature as above.
Because of the rescaling by $K_{\eps,N}^{-1}$, 
we finally get 
\[
\var \left[
K_{\eps,N}^{-1}\sum_{j=1}^N W_{i,j}^\eps   (x_i-X_j)
\right] = \Oh\left(  K_{\eps,N}^{-2}N^2 \frac{N}{\eps^2}K_{\eps,N}^{2}N^{-2} \eps  \right) = \Oh\left(   \frac{N}{\eps}  \right).
\]
Let us now turn to the covariance matrix of 
\[
V:= \vect \left( \sum_j W_{i,j}^\eps (x_i-X_j)(x_i-X_j)^\top\right)
\]
which, using Equations~(1.3.14), (1.3.16) and (1.3.31) in \citet{kollo2005advanced},  is equal to 
\[
N \expec [(W_j^\eps)^2 (x_i -X_j)\otimes(x_i -X_j)^\top\otimes (x_i -X_j)\otimes(x_i -X_j)^\top ]- N \expec V\expec^\top V.
\]

Relying on Lemma~\ref{lem: FourthOrd}, we get that the leading order of the variance of $K_{\eps,N}^{-1} V$ is 
\[
\Oh\left(K_{\eps,N}^{-2} N^2\ \frac{N}{\eps^2}\ K_{\eps,N}^{2} K_{\eps,N}^{\frac{d+4}{2}} \right)= \Oh\left( \frac{N K_{\eps, N}}{\eps}  \right).
\]
The claim follows.
\end{proof}

\subsection{Infinitesimal generator limit and spectral convergence}

A relatively general analysis of convergence of graph Laplacians was carried out by \citet{ting2011analysis}, wherein consistency results are established for a general class of constructions leveraging connections to diffusion processes. We remark that when $\mathcal{M}$ is endowed with a uniform measure, a constant approximation of the potential is valid and so the operator resulting from quadratically regularized optimal transport falls under their framework \citep[Theorem 3]{ting2011analysis}. The assumptions are compatible with the ones that we make here, namely that $\mathcal{M}$ is a smooth, compact manifold, and the authors consider a general kernel of the form $K_N(x, y) = w_x^{(N)}(y) K_0\left( \frac{\| y - x \|}{h_N r^{(N)}_x(y)} \right)$.

In our setting where iid samples are drawn uniformly on $\mathcal{M}$, we invoke a constant potential approximation $u \sim \varepsilon^{\frac{2}{2 + d}} N^{\frac{-4}{2 + d}}$, we have (up to a multiplicative constant)
\begin{align*}
    K_N(x, y) &= \left[ \varepsilon^{\frac{2}{2 + d}} N^{\frac{-4}{d + 2}} - \| y - x \|^2 \right]_+ = \left[ 1 - \left( \dfrac{\| y - x \|}{ \varepsilon^{\frac{1}{2+d}} N^{\frac{-2}{d + 2}} }\right)^2\right]_+ = \varphi\left( \frac{\|y - x \|}{h^{(N)}}\right)
\end{align*}
Where the choice of kernel is the Epanechnikov kernel $\varphi(r) = (1 - r^2)_+$. The condition under which their theorem holds  is that $N h^{m+2}/\log N \to \infty$. In our case, this simplifies to $\varepsilon/(N \log N) \to \infty$, and this is compatible with the range of scalings $N^{1 - 2/d} \ll \varepsilon \ll N^2$ from our previous analysis.

It is further possible understand how the eigenvalues and eigenvectors of the discrete operator relate to the continuous one, relying on the recent results by \citet{garcia2020error}. Their results apply in the setting that we consider for an intrinsic dimension $d\geq2$.
Upon choosing 
\[
\eps= N^{\frac{3d+2}{2d(d+2)}} (\log N)^{\frac{p_d(d+2)}{2}}, 
\]
with $p_d=3/4$ if $d=2$ and $p_d=1/d$ if $d\geq3$, the rate of convergence of the eigenvalues and eigenfunctions\footnote{We refer to the paper for an explicit description of how the eigenvector is interpolated to compute the norm between that interpolation and the eigenfunction on the manifold. } is
\[
\Oh\left(  \sqrt{\frac{(\log N)^{p_d}}{N^{1/d}}} \right)
\]
almost surely \citep[Theorems~1 and~5]{garcia2020error}.
%----------------------------------------------------------------------------
%
%                  Example Circle 
%
%----------------------------------------------------------------------------
\section{Equispaced points on the circle}

We finally consider an example for which the computations can be explicitly carried out: the case of equidistant points on the circle. 

\subsection{First-order optimal potentials}
\label{sec: EquiCirc}

We finish this section about the rates in the discrete case in a one-dimensional deterministic example.
Consider $N$ points that are equispaced on the circle each with mass $1/N$. 
Set $D_j$ to be the $j$-th squared Euclidean distance in the sorted list of all distances from one point to the others. 
We thus have 
\[
D_j = \left(2\sin\left(  \frac{\pi j}{N}\right)\right)^2= \frac{4\pi^2j^2}{N^2} + \Oh\left(\frac{j^4}{N^4} \right)
\]

We aim at solving 
\[
\sum_{j=1}^N (y - D_j )_+ = \frac{\varepsilon}{N}.
\]

We get, for $y$ small, that there exists $k$ such that 
\[
\frac{4\pi^2}{N^2} (k)^2\le  y \le \frac{4\pi^2}{N^2} (k+1)^2,
\]
that
\begin{align*}
    (2k+1)y - 2 \frac{4\pi^2}{N^2}\ \sum_{j=1}^k j^2 &= \frac{\eps}{N}.
\end{align*}

Thus, there exists $\alpha$ such that  
\begin{align*}
    (2k+1)(k+\alpha)^2 - 2\  \sum_{j=1}^k j^2 &= \frac{\eps N}{4\pi^2}\\
    (2k+1)(k+\alpha)^2 - \frac13  k(k+1)(2k+1) &= \frac{\eps N}{4\pi^2}.
\end{align*}

It follows that by matching the largest order for $k$
\[
\frac{4k^3}{3} \approx \frac{\eps N}{4\pi^2}
\]
and thus 
 \[
y \approx  \frac{4\pi^2}{N^2} \left( \frac{3\eps N }{16 \pi^2 }\right)^{2/3}
\]

Alternatively, to asses the quality of the approximation above,  consider 
\[
4\sin^2\left(  \frac{\pi k}{N}\right) \le  y < 4\sin^2\left(  \frac{\pi (k+1)}{N}\right),
\]
to derive 
\begin{equation}
\label{eq: Alternative}
(2k+1)4 \sin^2\left(  \frac{\pi (k+\alpha)}{N}\right) - 2 \sum_{j=1}^k 4 \sin^2\left(  \frac{\pi j}{N}\right) = \frac{\eps}{N}
\end{equation}
It further holds that 
\begin{align*}
\sum_{j=1}^k 4 \sin^2\left(  \frac{\pi j}{N}\right) &= 1+2k - \frac{\sin\left(  \frac{\pi (2k+1)}{N}\right)}{\sin(\pi/N)} \\
&= 1+2k - \frac{(2k+1)\pi /N - \frac{(2k+1)^3\pi^3 }{6N^3} + \Oh\big((2k+1)^5/N^5\big) }{\pi/N + \Oh(1/N^3) } \\
&= 1+2k - \frac{(2k+1)  - \frac{(2k+1)^3\pi^2 }{6 N^2} + \Oh\big((2k+1)^5/N^4\big) }{1 + \Oh(1/N^2) } \\
&=  \frac{(2k+1)^3\pi^2 }{6 N^2} + \Oh\left( \frac{2k+1}{N^2}\right) + \Oh\left(\frac{(2k+1)^5}{N^4}\right) \\
&=  \frac{4k^3\pi^2 }{3 N^2} + \Oh\left( \frac{k^2}{N^2}\right)  + \Oh\left(\frac{(2k+1)^5}{N^4}\right).
\end{align*}
 Plugging this result in \eqref{eq: Alternative}, one gets 
 \[
(2k+1)\left[4 \pi^2 (k+\alpha)^2 + \Oh\left(\frac{k^4}{N^4}\right) \right]   - 2 \frac{4k^3\pi^2 }{3 } + \Oh\left( k^2\right) + \Oh\left(\frac{(2k+1)^5}{N^2}\right)= \eps N
 \]
which gives 
\[
\frac{4k^3 }{3 } + \Oh\left( k^2\right)+ \Oh\left(\frac{k^5}{N^4}\right) + \Oh\left(\frac{(2k+1)^5}{N^2}\right)= \frac{\eps N}{4\pi^2}.
\]
This matches with the other approximation.
\subsection{Limiting operator}

Let us place ourselves in the same setting as Section~\ref{sec: EquiCirc} again. 
We have seen that the optimal potential must behave as $\kappa \eps^{2/3}N^{-4/3}$ at the first order.  

\begin{thm} 

Consider $g \in C^2$, $g: \reals^2 \to \reals$. Consider again a set $\{x_i\}_{i=1}^N$ of $N$ equispaced points on the unit circle. For simplicity, choose $i\le N$ such that $x_i=(0,1)$. Denote by $Q_i$, the Hessian of $g$ at $x_i$.  Then, defining 
 \[
( \Delta^{OT} g)(x_i) := \sum_{j=1}^N W_{i,j} \big( g(x_i) -g(x_j)\big), 
 \]
 with $W$ as above, there exist constants $C_1,C_2$, such that 
 \begin{align*}
C_1 \eps^{\tfrac{-2}{3}}N^{\tfrac{4}{3}} ( \Delta^{OT} g)(x_i)   &\to C_2 \frac{\partial g(z,y)} {\partial y} \Big\vert_{(z,y)=(0,1)}
+ \tr \left[ 
Q_i  
\begin{pmatrix} 0 & 0 \\ 0 & 1\end{pmatrix}
\right  ] \\
&=  C_2 \frac{\partial g(z,y)} {\partial y} \Big\vert_{(z,y)=(0,1)} + \frac{\partial^2 g(z,y)}{\partial z^2} \Big\vert_{(z,y)=(0,1)},
\end{align*}
for $ \lim_{N\to\infty } \eps N =\infty $.
\end{thm}
\begin{proof}
As $g$ is twice differentiable, we can write 
\[
g(x_j) = g(x_i) + L_i (x_i-x_j) + \frac12 (x_j -x_i)^\top Q_i  (x_j -x_i) + \oh\left ( \lVert x_i -x_j \rVert^2 \right),
\]
where $L_i$ is linear and $Q_i$ is the Hessian of $g$ evaluated at $x_i$.
Plugging this result in the definition of $( \Delta^{OT} g)(x_i)$, we derive 
\begin{align*}
&( \Delta^{OT} g)(x_i)  \\
& \qquad 
= \sum_{j=1}^N W_{i,j} \left( - L_i (x_i-x_j) - \frac12 (x_j -x_i)^\top Q_i  (x_j -x_i) + \oh\left ( \lVert x_i -x_j \rVert^2 \right)
\right).
\end{align*}
We will split this sum into three terms and control each one separately. 
Let us start with the the second term 
\[
- \frac12\sum_{j=1}^N W_{i,j}  (x_j -x_i)^\top Q_i  (x_j -x_i) =: B.
\]
The quantity in the above display is a scalar so that it is equal to its trace. Further, using the linearity and the cyclical property of the trace, it holds that
\begin{align*}
B &= - \frac12 \tr \left [ Q_i \sum_{j=1}^N W_{i,j}    (x_j -x_i) (x_j -x_i)^\top \right ].
\end{align*} 
Relabelling the points from closest to $i$ to furthest, computing explicitly $(x_j -x_i) (x_j -x_i)^\top$, we get 
\begin{align*}
B &= -  N \tr \left [ Q_i \sum_{j=1}^{\lfloor N/2\rfloor} \frac{(\kappa \eps^{2/3}N^{-4/3} - \tfrac{4\pi^2 j^2}{N^2} + \Oh(j^4/N^4) )_+}{\eps}  M_j   
\right ],
\end{align*} 
where \[
M_j := \begin{pmatrix} 
\vspace{2mm}
\sin^2(\frac{2\pi j}{N}) & - 2 \sin^2(\frac{\pi j}{N} ) \sin(\frac{2\pi j}{N} ) \\ 

- 2 \sin^2(\frac{\pi j}{N} ) \sin(\frac{2\pi j}{N} ) & 4 \sin^4(\frac{\pi j}{N}  ) \sin^2(\frac{2\pi j}{N}  )
\end{pmatrix}.
\]
A first order development gives 
\[
M_j \approx
\frac{4\pi^2}{N^2} 
\begin{pmatrix} 
\vspace{2mm}
j^2  &  -\frac{\pi}{2N}  j^3 \\ 

-\frac{\pi}{2N}  j^3  &  \frac{\pi^2 j^4}{N^2}   
\end{pmatrix}.
\]
Thus, 
\[
B \approx - \frac{4 \pi^2}{N}\tr \left[ 
Q_i  
E
\begin{pmatrix} 1 & 0 \\ 0 & 0\end{pmatrix}
\right  ],
\]
where 
\[
E:= 
\frac{4\pi^2}{\eps N^2}
\left( 
\frac{\kappa}{4\pi^2} (\eps N)^{2/3}  \left( \frac{1}{3} \left( \frac{\kappa}{4\pi^2}\right)^{3/2} \eps N + \Oh\left((\eps N)^{2/3}\right) \right) - 
\frac{1}{5} \left( \frac{\kappa}{4\pi^2}\right)^{5/2} (\eps N)^{5/3} +\Oh\left((\eps N)^{4/3}\right)
\right ) .
\]
So that  
\[
B \approx  \eps^{\tfrac{2}{3}}N^{\tfrac{-4}{3}} C_1   \tr \left[ 
Q_i  
\begin{pmatrix} 1 & 0 \\ 0 & 0\end{pmatrix}
\right  ] +  \Oh\left(\frac{\eps^{1/3} N^{-2/3} }{N} \right),
\]
for a constant $C_1$.
Let us now consider the Taylor residual.  It holds that 
\begin{align*}
 \sum_{j=1}^N W_{i,j}\ \oh\left ( \lVert x_i -x_j \rVert^2 \right) 
 & \approx  N \sum_{j=1}^N \frac{(\kappa \eps^{2/3}N^{-4/3} - \lVert x_i -x_j \rVert^2)_+}{\eps}\ \oh\left( \lVert x_i -x_j \rVert^2 \right)\\
  & \approx 2 N \sum_{j=1}^{\lfloor N/2 \rfloor} \frac{(\kappa \eps^{2/3}N^{-4/3} - \tfrac{4\pi^2 j^2}{N^2} + \Oh(j^4/N^4) )_+}{\eps}\ \oh\left( \frac{j^2}{N^2}\right).
\end{align*}
Comparing with the developments above, we see that the sum is a weighted sum of $\oh(j^2/N^2)$ whereas, for $B$ it was the sum of terms behaving like $j^2/N^2$ with the same weights. It follows that 
\[
\eps^{\tfrac{-2}{3}}N^{\tfrac{4}{3}} \sum_{j=1}^N W_{i,j}\ \oh\left ( \lVert x_i -x_j \rVert^2 \right)  =\oh(1).
\]
Let us finally address the first term, i.e., 
\[
-L_i \sum_{j=1}^N W_{i,j}   (x_i-x_j) ,
\]
Owing to the symmetry of the problem, this is constant times the normal vector at the point.
It is thus non zero in the $y$ direction. Remark that the vector $x_j - x_i$ has a component in the $y$ direction equal to $2\sin^2(\pi j /N  )$ and it thus has the same leading order as $M_j$ in terms of $\eps,N$. The claim follows.
\end{proof}
%----------------------------------------------------------------------------
%
%                  Simulations
%
%----------------------------------------------------------------------------
\section{Simulations}

In this section we exhibit the size of the optimal potentials obtained from the semismoothed Newton algorithm proposed in \citet{lorenz2021quadratically} and adapted to our setting in \citet{matsumoto2022beyond}.

\subsection{$d$-Sphere}
We now exhibit the behaviour of the optimal potentials for $N$ random points on the $d$-Sphere and various parameters $\eps$. In $d = 1, 2, 3$, $N = 10^3$ points were sampled uniformly by sampling from $S^d$ first a standard Gaussian and normalizing. We numerically solved the corresponding discrete optimal transport problem with $\varepsilon$ in the range $[10^{-3}, 10^5]$ and plotted $\log(\overline{u})$ against $\log(\varepsilon)$. For $\varepsilon$ sufficiently large, we estimated the exponent $\alpha$ for the relationship $u \sim \varepsilon^{\alpha}$. Our empirical findings agree with the exponent $\frac{2}{2 + d}$.
\begin{figure}[h!]
	\begin{center}
		\includegraphics[width=0.32\textwidth, ]{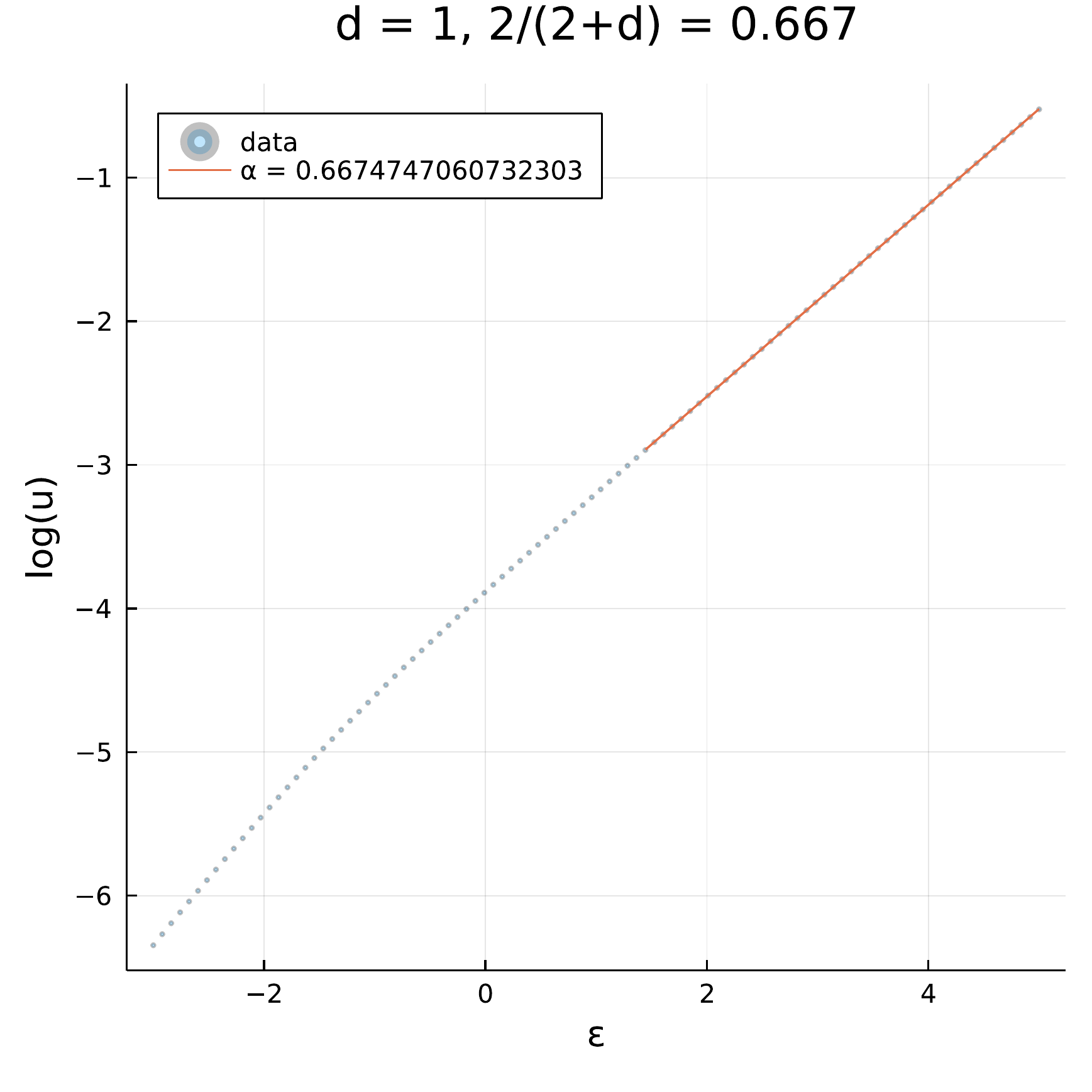}
		\includegraphics[width=0.32\textwidth, ]{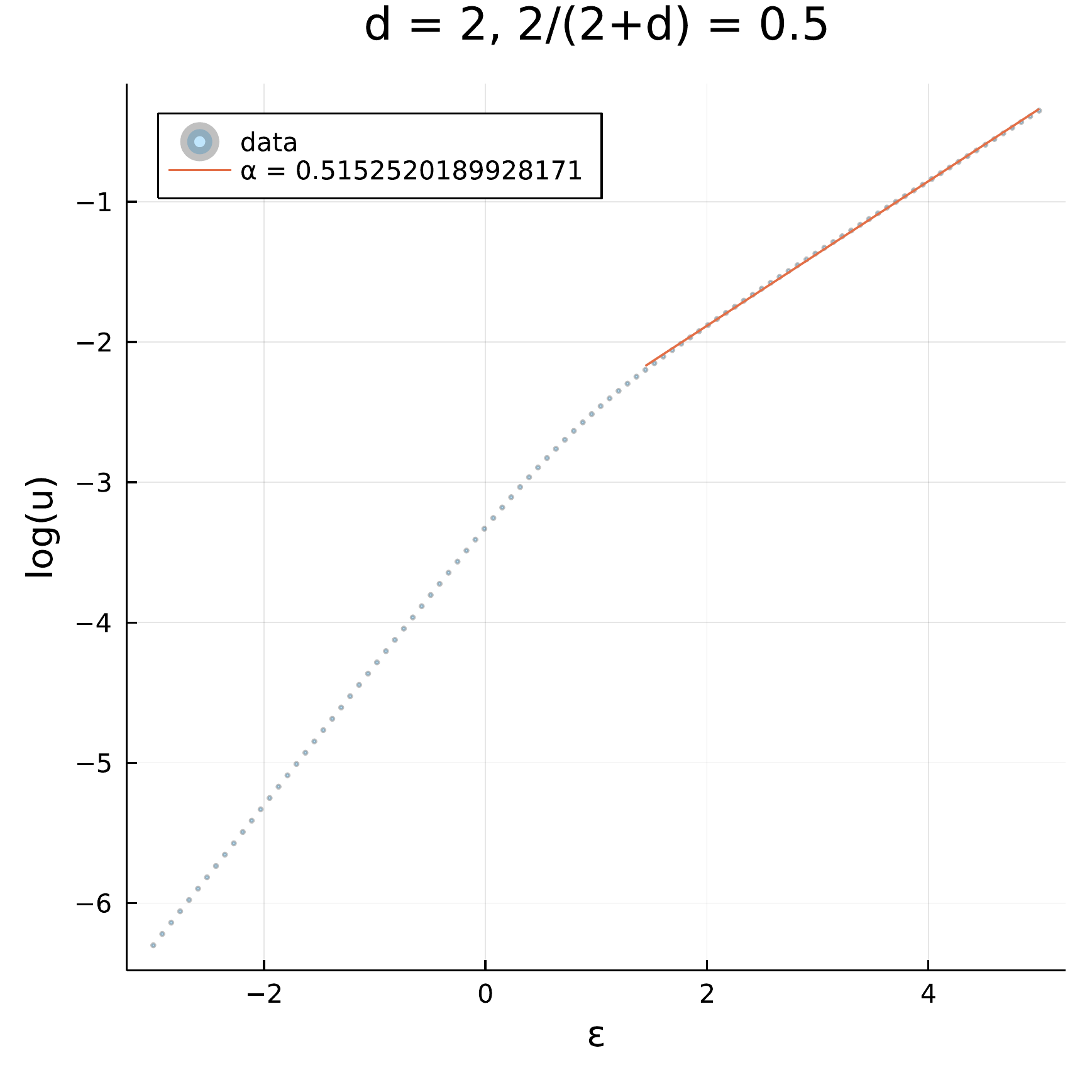}
		\includegraphics[width=0.32\textwidth, ]{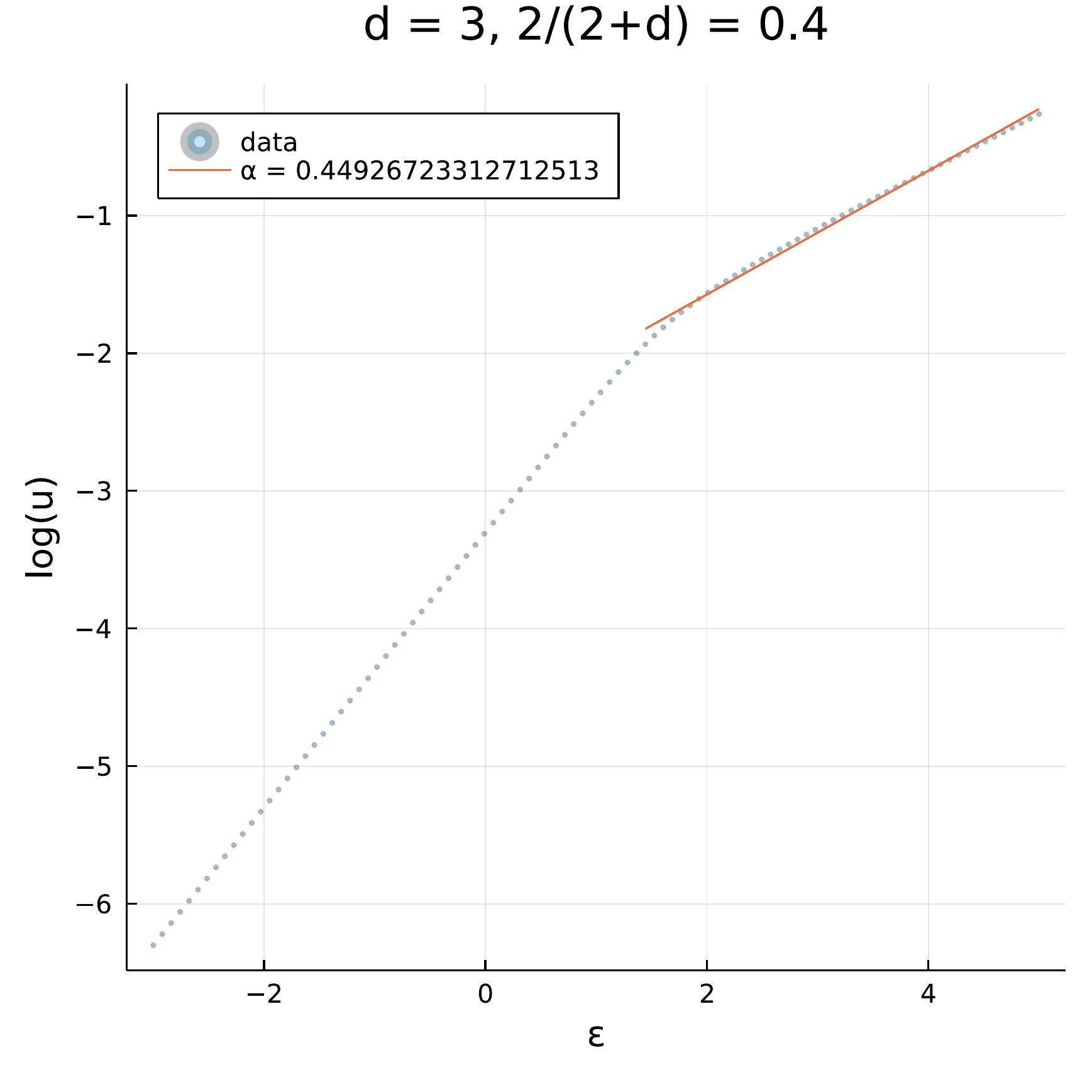}
	\end{center}
	\caption[]{Scaling of dual potential $u$ on the $d$-Sphere. }
	\label{fig:Sphere}
\end{figure}

\subsection{Torus}

Next, we investigate the behaviour of the operator $\Delta^{OT}$ in the discrete setting where points are sampled from the uniform distribution on the 2-dimensional torus with major and minor radii $R = 1, r = 1/2$. We fix a point $x_0 = (0, 1/2, 0)$ at which the tangent space $T_{x_0} \MC$ is spanned by $e_1, e_3$. We then consider a function $f(x, y, z) = 3x^2 + 5y^2 + 7z^2$. For $N = 100, 250, 500, 1000, 2500$ points sampled from the torus, we calculated the $(N+1) \times (N+1)$ coupling $\pi$ by solving \eqref{eq:qot}, normalized following \eqref{eq: GraphWeights}, and then computed the quantity $K_{\varepsilon, N}^{-1} (\Delta^{OT} f)(x_0)$. 

Motivated by the asymptotic scalings we derived, we tried setting $\varepsilon \propto N^\alpha$ for varying exponents: $\alpha = 2$ which should correspond to a fixed regularization level in the continuous case (and we do not expect convergence to the Laplacian in this case), and $\alpha = 1.05, 1.125, 1.25, 1.5, 1.75$ which all fall within the regime where Theorem 1 applies. We show in Figure \ref{fig:Torus} the values of $K_{\varepsilon, N}^{-1} (\Delta^{OT} f)(x_0)$ over 10 repeats at each value of $N$. 

We see that when $\alpha = 2$, the quantity $K_{\varepsilon, N}^{-1} (\Delta^{OT} f)(x_0)$ stabilizes around a fixed value as $N$ increases. This agrees with our understanding that $\varepsilon \propto N^2$ in the discrete setting corresponds to the continuous case of empirical distributions with a fixed value of $\varepsilon$. On the other hand, when $1 < \alpha < 2$ we observe a pattern of values appears to converge around a different value. Importantly, for various $1 < \alpha < 2$, these values are similar -- this supports the scaling relation of Theorem 1 and suggests that the quantity is converging to the value (up to a constant independent of $\varepsilon, N$) of the Laplace-Beltrami operator at $x_0$.

\begin{figure}[h!]
	\begin{center}
    	\includegraphics[width = 0.325\linewidth]{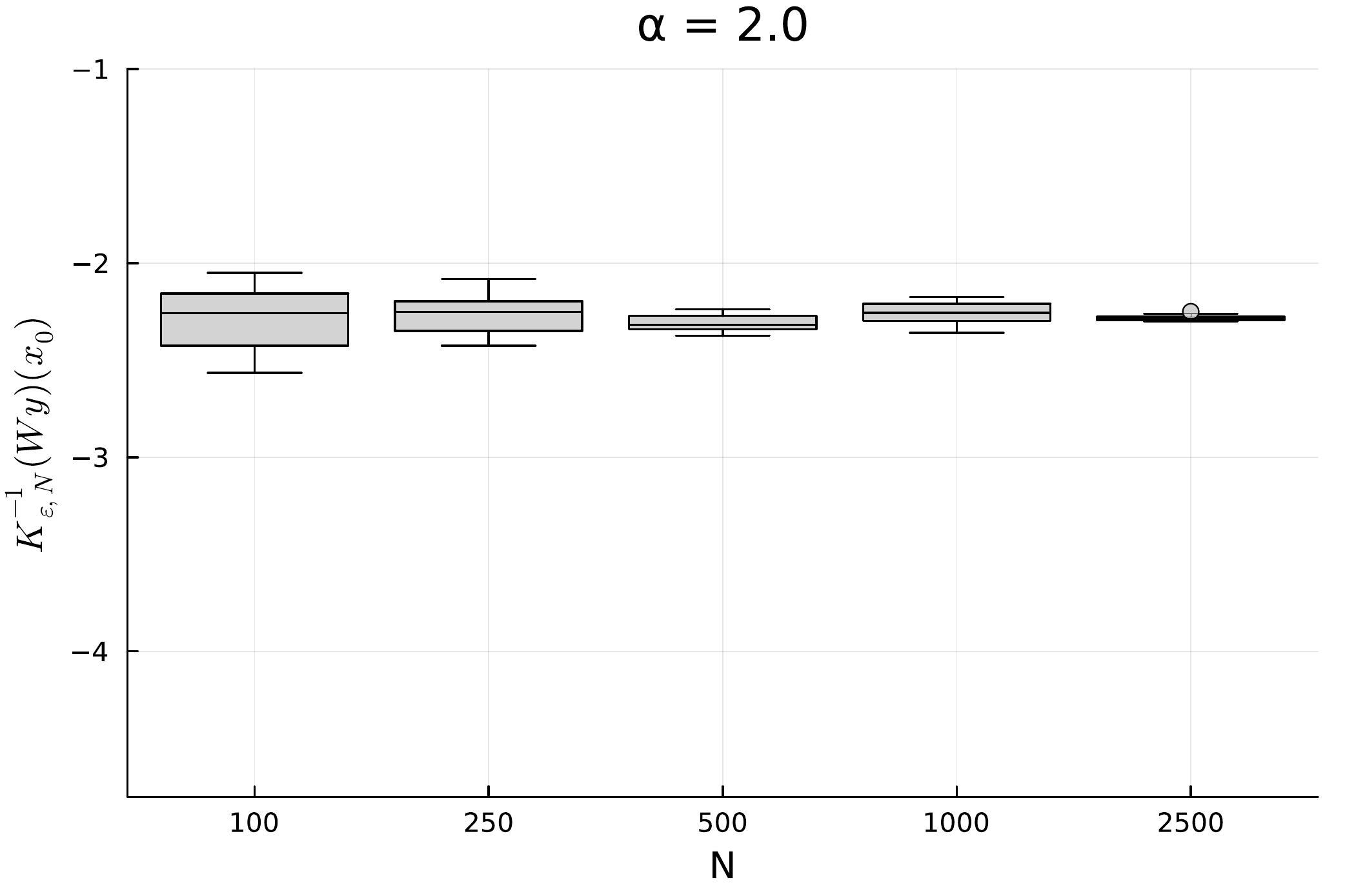}
    	\includegraphics[width = 0.325\linewidth]{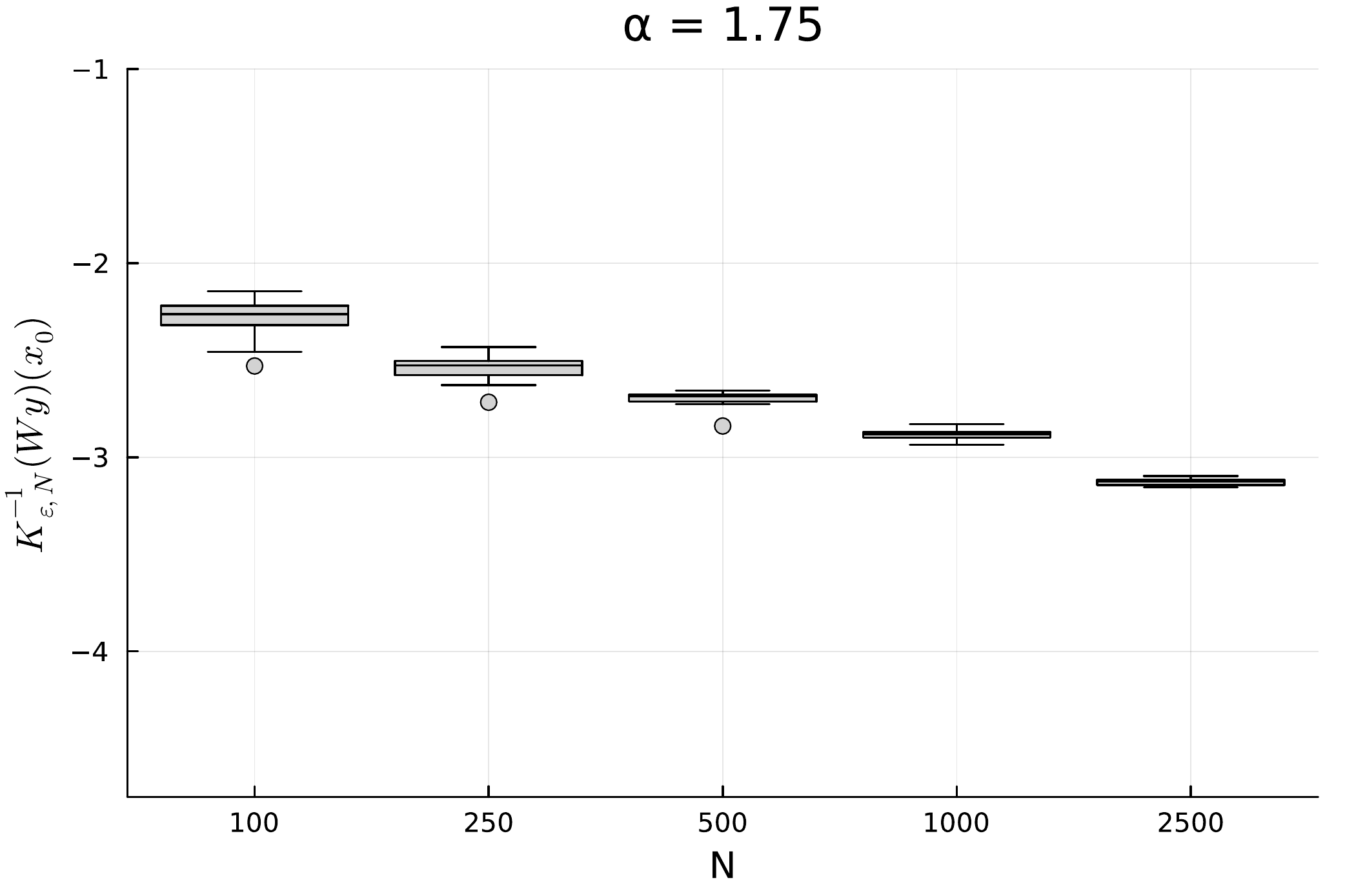}
    	\includegraphics[width = 0.325\linewidth]{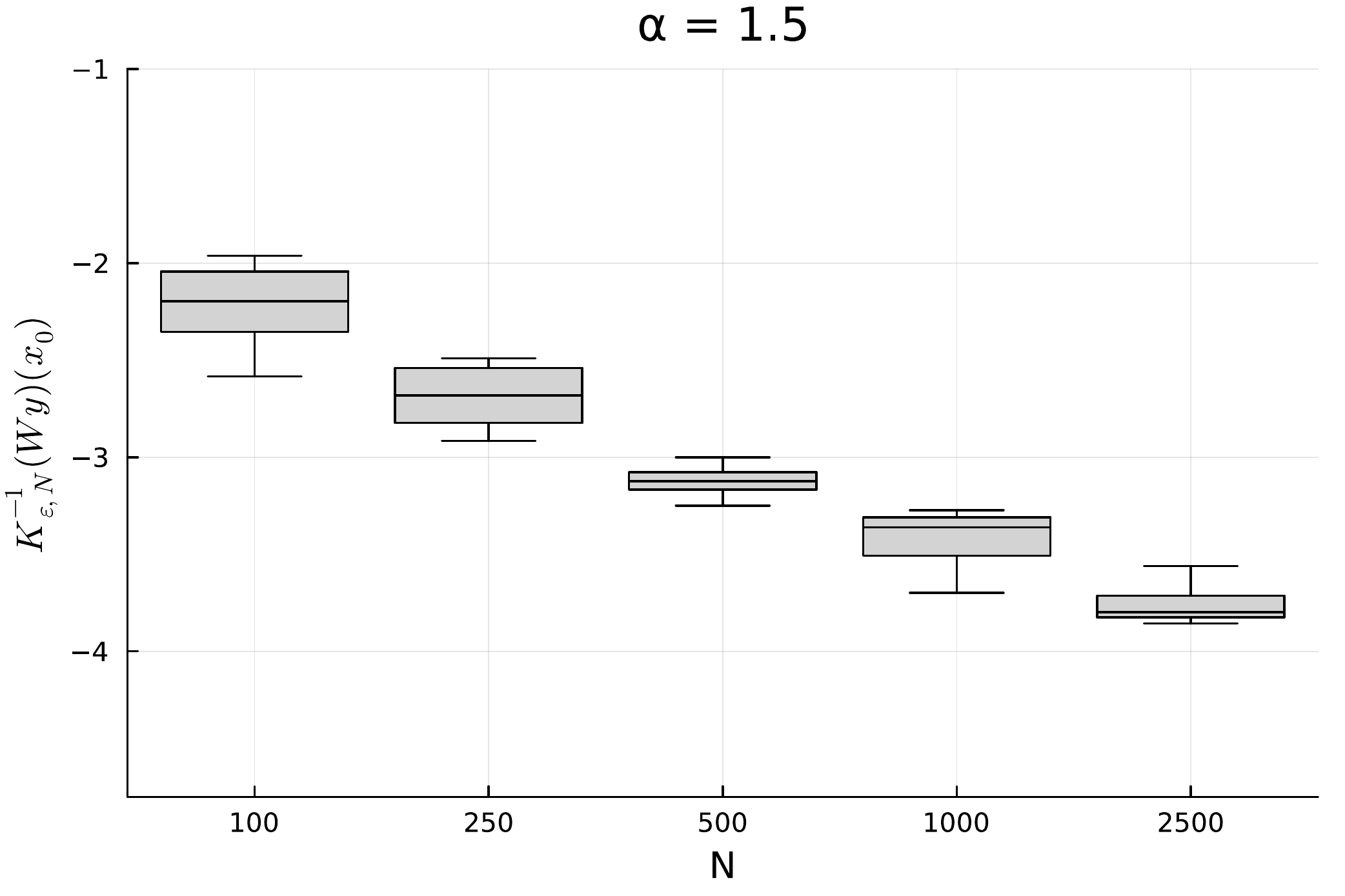} \\ 
    	\includegraphics[width = 0.325\linewidth]{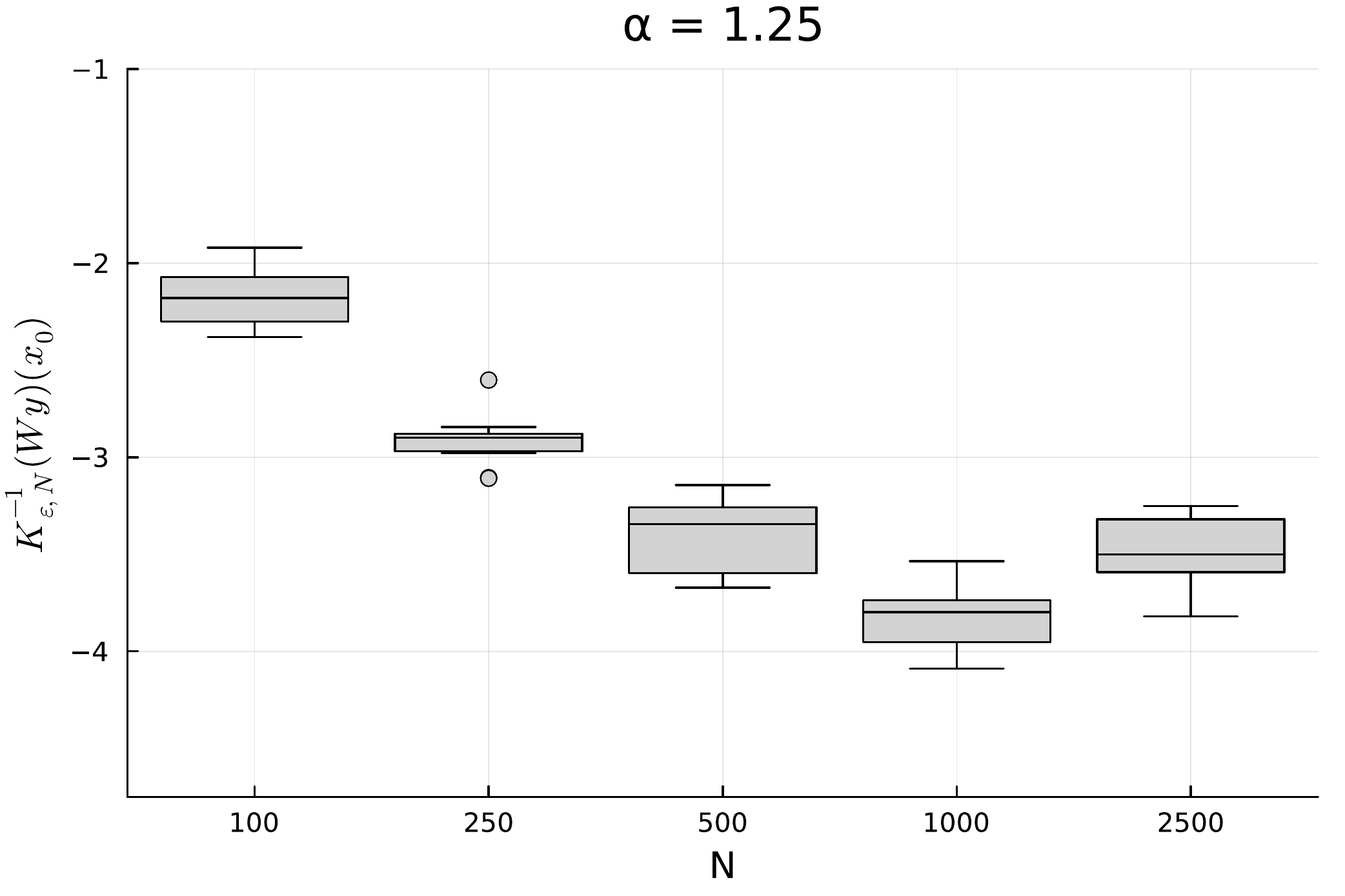} 
    	\includegraphics[width = 0.325\linewidth]{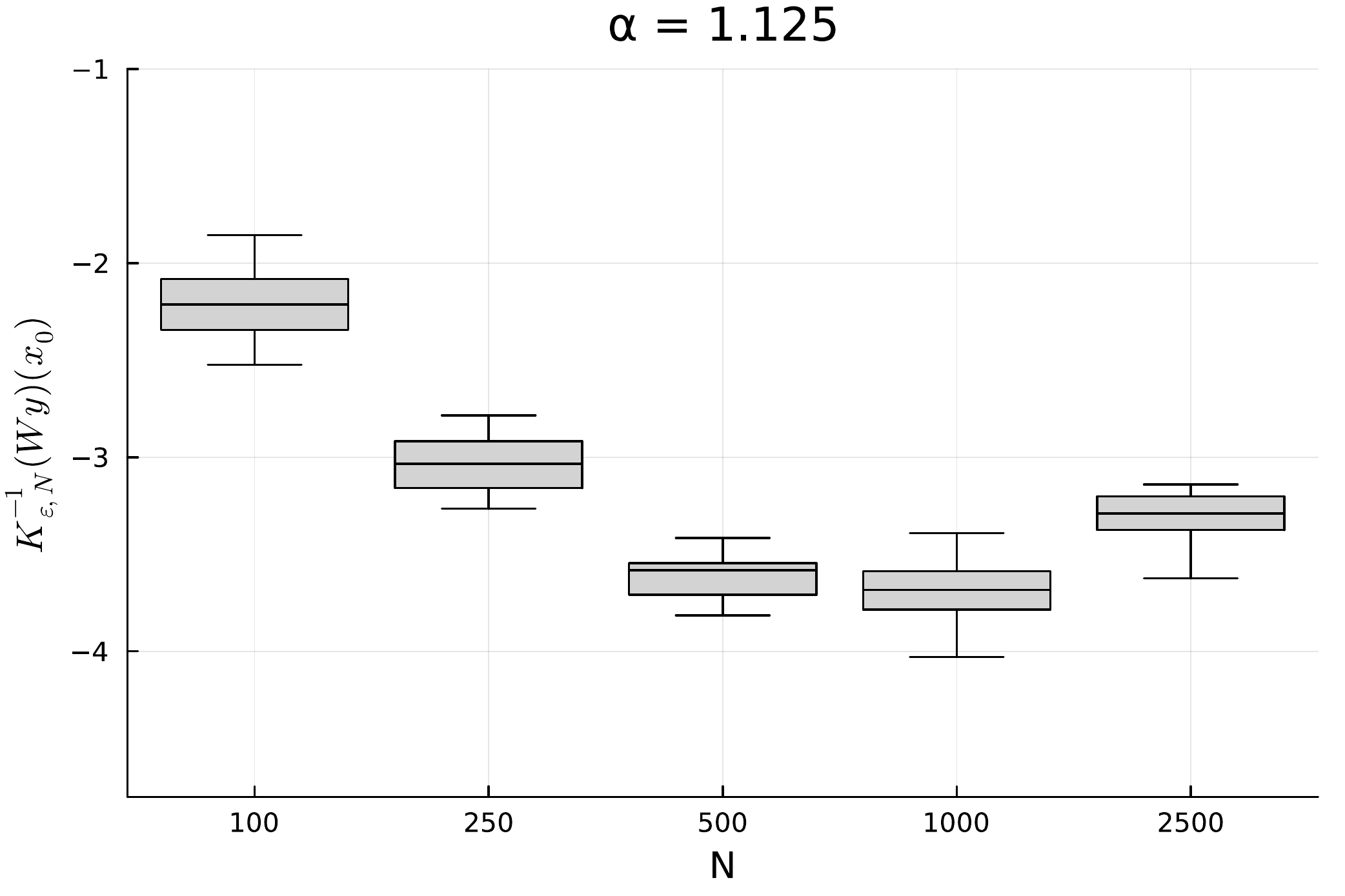}
    	\includegraphics[width = 0.325\linewidth]{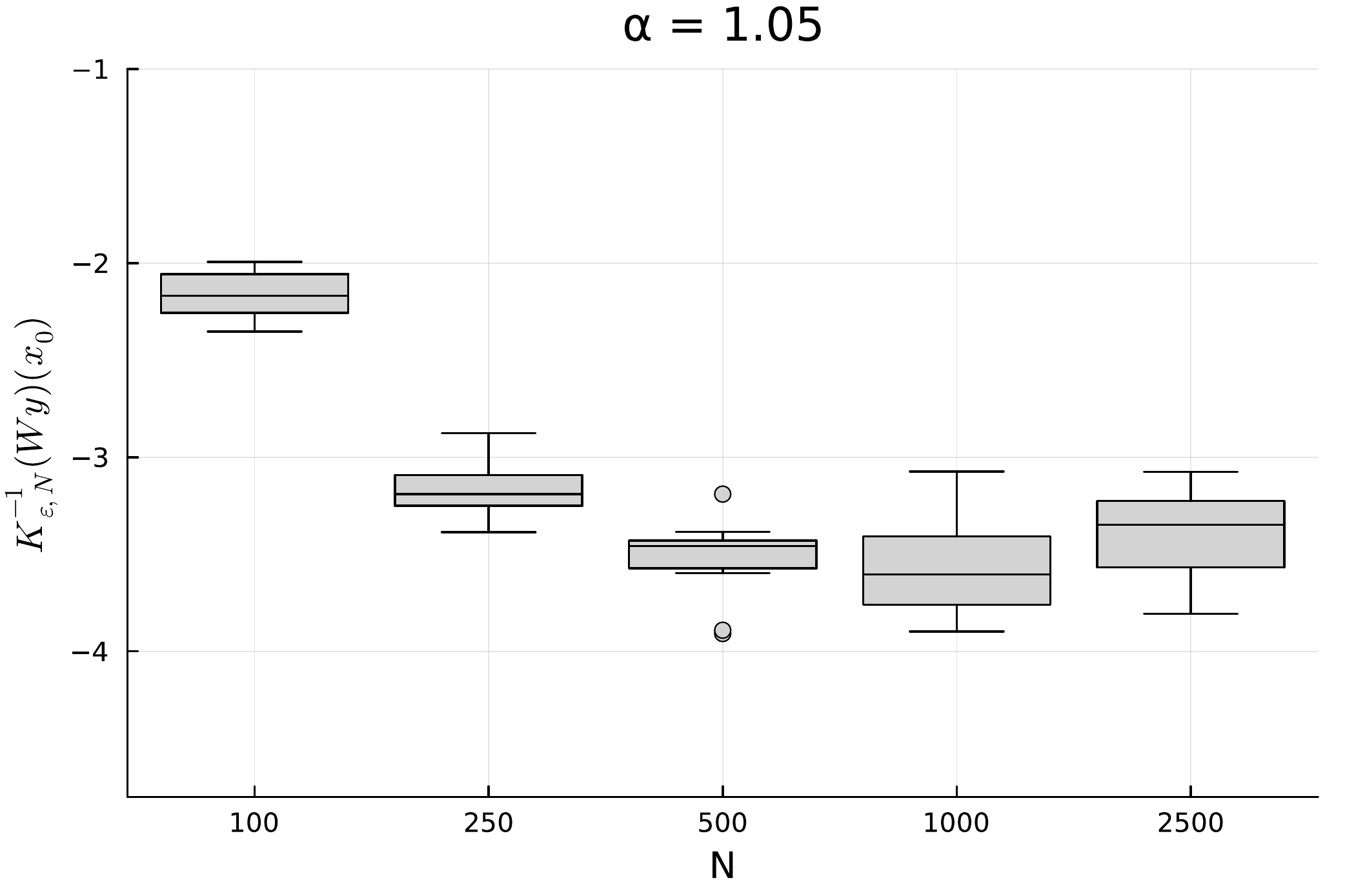}
	\end{center}
	\caption[]{Estimate of Laplacian $K_{\varepsilon, N}^{-1} (\Delta^{OT} f)(x_0)$ for varying sample sizes $N$, and $\varepsilon \propto N^\alpha$, for various choices $  \alpha = 2, 1.75, 1.5, 1.25, 1.125, 1.05$.}
	\label{fig:Torus}
\end{figure}

\printbibliography

\end{document}